\documentclass[11pt,a4paper]{amsart}
\usepackage{amssymb}
\usepackage{amsmath}
\usepackage{amsfonts,a4wide}
\usepackage{bbm}
\usepackage{enumerate}
\usepackage[latin1]{inputenc}
\usepackage{shadow}
\usepackage[all]{xy}
\usepackage{tikz-cd}

\newtheorem{thm}[subsection]{Theorem}
\newtheorem{defn}[subsection]{Definition}
\newtheorem{prop}[subsection]{Proposition}

\newtheorem{lemma}[subsection]{Lemma}

\newtheorem{remark}[subsection]{Remark}

\newcommand{\Z}{\mathbb{Z}}

\newcommand{\F}{\mathbb{F}}

\newcommand{\im}{\operatorname{Im}}

\DeclareMathOperator{\coker}{coker}

\DeclareMathOperator{\colim}{colim}

\DeclareMathOperator{\ind}{ind}

\begin{document}

\title[The cohomological index is not additive]
{The cohomological index of free $\Z/p$-actions is not additive with respect to join}

\author{Rafael Gomes}
\address{Department of Mathematics, Instituto Superior T\'ecnico,
Universidade de Lisboa, Av. Rovisco Pais, 1049-001 Lisboa, Portugal}
\email{rafael.m.gomes@tecnico.ulisboa.pt}

\author{Gustavo Granja}
\address{Center for Mathematical Analysis, Geometry and Dynamical Systems, Instituto Superior T\'ecnico,
Universidade de Lisboa, Av. Rovisco Pais, 1049-001 Lisboa, Portugal}
\email{gustavo.granja@tecnico.ulisboa.pt}

\maketitle

\begin{abstract} 
Let $p$ denote an odd prime. 
We show by example that the inequalities obtained in \cite{GKPS} for the behaviour of the cohomological index of a join of
free $\Z/p$-actions are sharp. Namely, for all odd integers $k,l$ at least one of which is greater than one, we give examples 
of finite free $\Z/p$-CW complexes of cohomological indices $k$ and $l$ whose join has index $k+l+1$ and also examples 
where the join has index $k+l-1$.
\end{abstract}

\section{Introduction}

Let $p$ be an odd prime, $\tilde X$ be a free $\Z/p$-CW complex and $X= \tilde X/\Z/p$ be its orbit space. Let $f\colon X \to B\Z/p$
be a map classifying the covering $\tilde X \to X$. The \emph{cohomological index} of the free $\Z/p$-action
is the dimension of the $\F_p$-vector space $\im f_* \subset H_*(B\Z/p; \F_p)$.  We will use the notation $\ind(X)$ 
for the cohomological index so as not to clash with the notation in \cite{GKPS}. It will always be clear from the context
 which free action having $X$ as an orbit space we are considering.

Using the product structure on $H^*(B\Z/p)$ and the action of the Bockstein operation, it is easy to check that $\im f_*$ is of 
the form $\oplus_{i=0}^{k-1} H_i(B\Z/p)$ for some $0\leq k \leq \infty$ (i.e. there are no "gaps" in the image of $f_*$). Hence
one may also define $\ind(X)$ as the least possible degree of a non-zero class in $\ker f^* \colon H^*(B\Z/p) \to H^*(X)$.

Cohomological indices of this type for free $G$-spaces were introduced and studied by 
Fadell and Rabinowitz \cite{FR}. Hara and Kishimoto \cite{HK} call $\ind(X)-1$ the \emph{height} of the covering $\tilde X \to X$. 
The height is a lower bound for the category (in the sense of Lusternik-Schnirrelman) 
of the classifying map $f\colon X \to B\Z/p$ (see \cite[Section 3]{HK}).

Let $\tilde X$ and $\tilde Y$ be free $\Z/p$-CW complexes. The join $\tilde X \ast \tilde Y$ has a natural diagonal $\Z/p$-action and 
we write 
$$X \ast_p Y = (\tilde X \ast \tilde Y)/\Z/p$$
for the orbit space. This paper is concerned with the relation between 
$\ind(X \ast_p Y)$ and the cohomological  indices of $X$ and $Y$.

The following\footnote{The paper \cite{GKPS} dealt only with free $\Z/p$-actions which can be embedded in an odd sphere with 
a free linear $\Z/p$-action but it is easy to see that no loss of generality results from making that assumption.}
is known regarding this question. 

\begin{thm}{\cite[Proposition 3.9]{GKPS}}
\label{oldaddit}
Let $\tilde X$ and $\tilde Y$ be free $\Z/p$-CW complexes. Then  
$$ 
\ind(X \ast_p Y) = \ind(X) + \ind(Y) \ \text{ when either } \ind(X) \text{ or } \ind(Y)  \text{ are even}
$$
When $\ind(X)$ and $\ind(Y)$ are both odd, the following inequalities hold
$$ \ind(X) + \ind(Y) -1 \leq \ind(X \ast_p Y) \leq \ind(X) + \ind(Y) + 1. $$
\end{thm}

The purpose of this note is to show that the inequalities in the previous statement are sharp by giving examples of 
$\Z/p$-free CW complexes realizing the upper and lower bounds for any given pair of odd $\ind(X)$ and $\ind(Y)$ 
not both equal to one. The case where both indices are one is trivial (and additivity holds).

Our motivation comes from contact topology. In \cite{GKPS}, following Givental \cite{Gi}, the cohomological index was used
 to construct an integer valued quasi-morphism on the universal cover of the contactomorphism group of a lens space 
 (with its standard contact structure). If the cohomological index were exactly additive, the arguments in that paper would 
 prove Sandon's contact version of the Arnold conjecture for lens spaces. Unfortunately additivity does not hold, so we are not 
 able to improve on the lower bounds in \cite{GKPS} for the number of translated points of a contactomorphism of a lens space 
 which is contact isotopic to the identity. 

\subsection{Notation and conventions.} Homology and cohomology is with $\F_p$ coefficients for $p$ an odd prime, unless 
otherwise noted. We will generally denote free $\Z/p$-CW complexes by letters decorated with a tilde and remove the 
tilde to denote their orbit spaces. We will write $\{\ast\}$ for a one point space, the orbit space of $\Z/p$.

We write $R= \F_p[\Z/p]$ for the group algebra of $\Z/p$, $g$ for the canonical generator of $\Z/p$ and $\tau=g^0-g \in R$ so
that $R= \F_p[\tau]/\tau^p$. The indecomposable finitely generated $R$-modules are the submodules $\tau^i R \subset R$ for $i=0,\ldots,p-1$. 
Any finitely generated $R$-module is isomorphic to a direct sum of indecomposable submodules.  

$C_*(\tilde X)$ denotes the cellular chain complex of the free $\Z/p$-CW complex $\tilde X$ with $\F_p$ 
coefficients. This is a chain complex of free $R$-modules. $C_*(X)$ denotes the cellular chain complex of the orbit space also 
with $\F_p$-coefficients. Note that $C_*(X) = C_*(\tilde X) \otimes_R \F_p$ where $\F_p$ is regarded as an $R$-algebra via
 the augmentation (which sends $\tau$ to $0$). 

\subsection*{Acknowledgements} The first author was partially supported by FCT/Portugal through the fellowship BL236/2017. The second author was partially supported by FCT/Portugal through project UIDB/MAT/04459/2020. 

\section{Generalities on the homology of a join}

If $X$ and $Y$ are CW complexes, we give $X \ast Y= (X \times [0,1] \times Y)/\sim$ its natural CW structure. 
This has $X$ and $Y$ as subcomplexes, while
the remaining cells are of the form $e \ast f$ for $e$ a cell of $X$ and $f$ a cell of $Y$. If $e$ is an $m$-cell and $f$ an $n$-cell,
the boundary map of the generator $e\ast f$ in $C_*(X \ast Y)$ is given by the expression
$$
\partial(e\ast f) = \begin{cases}
f-e & \text{ if } m=n=0 \\
f - e\ast (\partial f)  & \text{ if } m=0, n>0 \\
\partial(e) \ast f + (-1)^{m+1} e & \text{ if } m>0, n=0 \\
\partial(e) \ast f + (-1)^{m+1} e \ast \partial(f) & \text{ otherwise.}
\end{cases}
$$
When $M$ and $N$ are $R$-modules, the vector space $M \otimes_{\F_p} N$ has a natural diagonal $\Z/p$-action and therefore
becomes an $R$-module. In terms of the presentation of $R$ as $\F_p[\tau]/\tau^p$, the module structure on a decomposable 
element of $M \otimes_{\F_p} N$ is described by
\begin{equation}
\label{actiontau}
\tau(a\otimes b) = (\tau a) \otimes b + a \otimes (\tau b) - (\tau a) \otimes (\tau b)
\end{equation}
Consequently, for $\tilde X$ and $\tilde Y$ free $\Z/p$-CW complexes, the action of $\tau$ on an element of the form
$e\ast f \in C_*(\tilde X \ast \tilde Y)$ is given by 
$$\tau(e \ast f) = (\tau e) \ast f +e\ast(\tau f) - (\tau e)\ast (\tau f).$$

\subsection{Homology of $X \ast_p \{\ast\}$}
Let $\tilde X$ be a free $\Z/p$-CW complex.  The join $\tilde X \ast \Z/p$ is a union of $p$ cones on a 
common base $\tilde X$, hence
$$
\tilde X \ast \Z/p \simeq \vee_{i=1}^{p-1} \Sigma \tilde X
$$
\begin{lemma}
\label{joinZp}
Let $\tilde X$ be a free $\Z/p$-CW complex. For each $k\geq 1$ we have the following
isomorphism of $R$-modules
$$ H_{k+1} (\tilde X \ast \Z/p ) \cong H_k(\tilde X) \otimes_{\F_p} \tau R $$
(where the tensor product is given the $R$-module structure described by \eqref{actiontau}).
\end{lemma}
\begin{proof}
Given a cellular cycle $z \in C_k(\tilde X)$, the isomorphism 
$$H_k(\tilde X) \otimes_{\F_p} \tau R \to H_{k+1}(\tilde X \ast \Z/p)$$
 sends  $[z] \otimes \tau$ to the homology class $[z \ast g^0 - z \ast g] \in H_{k+1}(\tilde X \ast \Z/p)$.
\end{proof}
For later use we record the following computation which follows easily from \eqref{actiontau} (see for 
instance \cite[Section 2.3]{Be}). 
\begin{lemma}
\label{tensors}
For the $R$-module structure on the tensor product described by \eqref{actiontau}
we have 
$$
\tau^{p-1}R \otimes_{\F_p} \tau R \cong \tau R, \quad \quad \tau^{p-2} R \otimes_{\F_p} \tau R \cong \tau^2 R \oplus R
$$
\end{lemma}
The following observation follows immediately from the definitions.
\begin{lemma}
\label{joinpt}
Let $\tilde X$ be a free $\Z/p$-CW complex. There is a cofiber sequence
$$
\tilde X \xrightarrow{q} X \to X \ast_p \{\ast\}
$$
and hence a short exact sequence
$$
0 \to \coker q_* \to \tilde H_*(X \ast_p \{\ast\}) \to \Sigma\ker q_* \to 0 
$$
\end{lemma}

\section{Counterexamples to additivity of the index}

We consider the standard model for the universal principal $\Z/p$ bundle 
$$  S^\infty \to L^\infty_p $$
where $\Z/p$ acts diagonally on $S^\infty = \colim S^{2n+1}$. The infinite lens space $L^\infty_p$ has a standard cell decomposition 
(described for instance in \cite[Example 2.43]{Ha}) with one cell in each dimension. This in turn gives rise to a $\Z/p$-free cell 
decomposition of $S^\infty = \tilde{L^\infty_p} = E\Z/p$ whose cellular chain complex $C_*(\tilde L^{\infty}_p)$ 
is the standard (periodic) resolution of $\F_p$ over $R$:
$$
\cdots \to R \xrightarrow{\tau^{p-1}}  R \xrightarrow{\tau} R \xrightarrow{\tau^{p-1}}  R \xrightarrow{\tau} R
$$
We write $\tilde L^k_p$ for the $k$-skeleton of 
this decomposition and note that 
$$
\tilde L^k_p = \begin{cases}
S^k & \text { if } k \text{ is odd } \\
S^{k-1} \ast \Z/p & \text{ if } k  \text{ is even,}
\end{cases}
$$
so that 
$$
L^{2m}_p = L^{2m-1}_p \ast_p \{\ast\}
$$
The actions on the skeleta $\tilde L^k_p$ are the easiest examples of actions with index $k+1$. 
Using the associativity and commutativity of the 
join up to natural homeomorphism (see \cite[Section B.2]{GKPS}) it is easy to see that the index is additive on skeleta of lens spaces as
$$
L^k_p \ast_p L^l_p \cong \begin{cases}
L^{k+l+1}_p & \text{ if } k \text{ or } l \text{ are odd } \\
L^{k+l-1} \ast_p ( \{\ast\} \ast_p \{\ast\})  & \text{ if both } k \text{ and } l \text{ are even} 
\end{cases}
$$
and one checks easily that $\ind( \{\ast\} \ast_p \{\ast\})=2$. 

We will now give examples in which the index is not additive. We first note that in order to produce examples in all possible cases 
 (i.e. for arbitrary odd indices of the factors, not both equal to one), it suffices to give examples of spaces $X$ with index $3$
such that additivity does not hold for $X \ast_p \{\ast\}$:
\begin{lemma}
Let $k\geq 1$ and $l\geq 0$ be integers. Let $\tilde X$ be a free $\Z/p$-CW complex and
 let $Z= X\ast_p L^{2k-1}_p$. Then additivity holds for the index of $Z\ast_p L^{2l}$ if and only if it holds for
 the index of $X \ast_p \{\ast\}$.
\end{lemma}
\begin{proof}
Since the join is associative and commutative up to natural homeomorphism we have 
$$
Z \ast_p L^{2l}  \cong X \ast_p L^{2k-1} \ast_p L^{2l-1} \ast_p \{\ast\} \cong (X \ast_p \{\ast\}) \ast_p L^{2k+2l-1}
$$
It follows from Theorem \ref{oldaddit} that 
$$ \ind(Z \ast_p L^{2l}) = \ind(X \ast_p \{\ast\}) + 2k + 2l $$
Since, again by Theorem \ref{oldaddit}, we have $\ind(Z)= \ind(X) + 2k$ and $\ind(L^{2l}) = \ind( \{\ast\}) + 2l$  the statement follows.
\end{proof}

Recall that $\tilde L^2 = S^1 \ast \Z/p \simeq\vee_{i=1}^{p-1} S^2$. As a $\Z[\Z/p]$-module, $H_2(\tilde L^2;\Z)$ is
 naturally the kernel of the augmentation homomorphism $\Z[\Z/p] \to \Z$. 
\begin{defn}
\label{upper}
Let $\tilde U$ be the free $\Z/p$ cell complex obtained from $\tilde L^2$ by attaching a free $\Z/p$ $2$-cell by a map 
$S^2 \xrightarrow{a} \tilde L^2$ with $a_*([S^2]) = g^0-2g+g^2 \in H_2(\tilde L^2;\Z)$. 
\end{defn}
Note that the attaching map $a$ in the previous definition exists as the Hurewicz homomorphism $\pi_2(\tilde L^2) \to H_2(\tilde L^2; \Z)$
is an isomorphism.  By definition, $\tilde U$ sits in a cofiber sequence 
$$
\Z/p \times S^2 \xrightarrow{\tilde a} \tilde L^2 \to \tilde U
$$
where $\tilde a$ is the equivariant extension of the map $a$.

We will make use of the following observation whose proof is left 
to the reader. 
\begin{lemma}
\label{transg}
Let $\tilde X$ be an $(m-2)$-connected free $\Z/p$-cell complex with $m\geq 2$. The differential $d_m \colon E_{m,0}^m \to E_{0,m-1}^m$ in
the Serre spectral sequence of the homotopy fiber sequence
$$
\tilde X \to X \to B\Z/p
$$
is trivial if and only if $\ind(X)\geq m+1$.
\end{lemma}

\begin{prop}
\label{propupper}
Let $\tilde U$ be the space of Definition \ref{upper}. Then 
$$\ind(U)=3 \text{  and }\ind(U \ast_p \{\ast\})= 5.$$
\end{prop}
\begin{proof}
Since $\tau^2=g^0-2g+g^2 \in R$, the chain complex $C_*(\tilde U)$ is given by
$$  \cdots \to 0 \to R \xrightarrow{\tau^2} R \xrightarrow{\tau^{p-1}} R \xrightarrow{\tau} R $$
and hence, as $R$-modules we have
$$
H_k(\tilde U) \cong \begin{cases}
\tau^{p-1}R & \text{ if } k= 0,2 \\
\tau^{p-2}R & \text{ if } k= 3 \\
0 & \text{ otherwise.}
\end{cases}
$$
From Lemmas \ref{joinZp} and \ref{tensors} we conclude that 
\begin{equation}
\label{cohjoin1}
H_k(\tilde U \ast \Z/p) \cong \begin{cases}
\tau^{p-1} R & \text{ if } k=0 \\
\tau R & \text{ if } k= 3 \\
\tau^2R \oplus R & \text{ if } k= 4 \\
0 & \text{ otherwise.}
\end{cases}
\end{equation}

On the other hand, $C_*(U) = C_*(\tilde U) \otimes_R \F_p$ is the complex
$$ \cdots \to 0 \to \F_p \xrightarrow{0} \F_p \xrightarrow{0} \F_p \xrightarrow{0} \F_p $$
hence the projection map $q \colon \tilde U \to U$ induces the $0$ homomorphism in positive degrees. 
Lemma \ref{joinpt} implies that 
$$
\tilde H_k( U \ast_p \{\ast\}) = \begin{cases}
\F_p & \text{ if } k=1,2 \\
\F_p \oplus \F_p & \text{ if } k=3,4 \\
0 & \text{otherwise.}
\end{cases}
$$

Now consider the Leray-Serre spectral sequence of the Borel construction on $\tilde U \ast \Z/p$
\begin{equation}
\label{SSS1}
\tilde U \ast \Z/p  \xrightarrow{i} (\tilde U \ast \Z/p)_{h\Z/p}  \to B\Z/p
\end{equation}
and recall that the natural projection $(\tilde U \ast \Z/p)_{h\Z/p} \xrightarrow{\pi} U\ast_p \{\ast\}$ is a homotopy equivalence making the 
following diagram commute:
$$
\begin{tikzcd}
\tilde U \ast \Z/p \arrow[r,"i"] \arrow[rd,"q"] &  (\tilde U \ast \Z/p)_{h\Z/p} \arrow[d,"\pi"] \arrow[r]  & B\Z/p\\
& U\ast_p \{\ast\} \arrow[ur] & 
\end{tikzcd}
$$
We have 
$$ E^2_{p,q} = H_p(B\Z/p; \{ H_q( \tilde U \ast \Z/p) \}) = H_p(\Z/p; H_q(\tilde U \ast \Z/p))  \Rightarrow H_{p+q}( U\ast_p \{\ast\})$$
In view of \eqref{cohjoin1}, we see that $E^2_{p,q}=0$ for $q=1,2$ and $E^2_{0,3} \cong \F_p$. As $H_3(U \ast_p \{\ast\}) \cong
\F_p^2$ it follows that the differential $d^4 \colon E^4_{4,0} \to E^4_{0,3}$ must be $0$ and hence, by Lemma \ref{transg}, we have 
$\ind(U \ast_p \{\ast\}) \geq 5$. As $\dim U \ast_p \{\ast\}$=4, it follows that 
$\ind(U \ast_p \{\ast\})=5$.

We leave the similar, but easier, argument showing that $\ind(U)=3$ to the reader. 
\end{proof}

\begin{remark}
Let $k\geq 1$, $\alpha_k$ denote a generator of $H^k(B\Z/p)$ and $Y_k$ denote the homotopy fiber of 
$B\Z/p \xrightarrow{\alpha_k} K(\Z/p,k)$. Then $Y_k$ is the orbit space of a free $\Z/p$-action on a complex 
with the homotopy type of $K(\Z/p,k-1)$. This action is a ''versal" action of index $k$ in the sense 
that, for any free $\Z/p$-CW complex $\tilde X$ with $\ind(X)=k$, the classifying map $X \to B\Z/p$ factors through $Y_k$.

It was shown in \cite{Go} that for $k,l$ odd, not both equal to one, $\ind(Y_k \ast_p Y_l)=k+l+2$. The space $\tilde U$ of Definition \ref{upper}
is the essential part of the $3$-skeleton of a $\Z/p$-free cellular approximation of $K(\Z/p,2) = \tilde Y_3$.
\end{remark}

\begin{lemma}
\label{lower}
There exists a two dimensional free $\Z/p$-cell complex $\tilde V$ such that $C_*(\tilde V)$ is isomorphic to the complex
$$
\cdots \to 0 \to R \xrightarrow{ \left[ \begin{array}{c} \tau^{p-1} \\ -\tau \end{array} \right] } R \oplus R 
\xrightarrow{ \left[ \begin{array}{cc} \tau & 0 \end{array} \right]} R 
$$
\end{lemma}
\begin{proof}
The $1$-skeleton of $\tilde V$ is obtained by attaching two $1$-cells
$e_1,e_2$ to $\tilde V^0 = \Z/p$ using the attaching map of the $1$-cell in $\tilde L^1_p=S^1$. 
We define the isomorphism $\psi \colon R \oplus R \to C_1(\tilde V)$ by 
$\psi(g^0,0)=e_1$ and $\psi(0,g^0)=e_1-e_2$.

Since $\tilde V^1$ is homotopy equivalent to a wedge of circles, it is clear we can pick an attaching map $S^1 \to \tilde V^1$ 
so as to realise the desired differential $\partial_2$.
\end{proof}

\begin{prop}
\label{proplower}
Let $\tilde V$ be one of the spaces provided by Lemma \ref{lower}. Then 
$$\ind(V)=3 \text{  and }\ind(V \ast_p \{\ast\})= 3.$$
\end{prop}
\begin{proof}
The proof is analogous to that of Proposition \ref{propupper} so we just indicate the main steps. We have 
$$
H_k(\tilde V) \cong \begin{cases}
\tau^{p-1}R & \text{ if } k= 0, 2\\
\tau^{p-2} R & \text{ if } k= 1 \\
0 & \text{ otherwise.}
\end{cases}
$$
From Lemmas \ref{joinZp} and \ref{tensors} we conclude that 
$$
H_k(\tilde V \ast \Z/p) \cong \begin{cases}
\tau^{p-1}R & \text{ if } k= 0 \\
\tau^2R \oplus R & \text{ if } k= 2 \\
\tau R & \text{ if } k= 3 \\
0 & \text{ otherwise.}
\end{cases}
$$

The complex $C_*(V)$ is isomorphic to
$$ \cdots \to 0 \to \F_p \xrightarrow{0} \F_p \oplus \F_p \xrightarrow{0} \F_p $$
and Lemma \ref{joinpt} implies
$$
\tilde H_k( V \ast_p \{\ast\}) = \begin{cases}
\F_p & \text{ if } k=1,3 \\
\F_p \oplus \F_p & \text{ if } k=2 \\
0 & \text{otherwise.}
\end{cases}
$$

In the Leray-Serre spectral sequence of 
$$
\tilde V \ast \Z/p  \to V \ast_p \{\ast\}  \to B\Z/p
$$
there must be a differential $d^3 \colon E^3_{3.0} \to E^3_{0,2} \cong \F_p^2$, for otherwise we would have
$H_2(V \ast \Z/p) \cong \F_p^3$. It follows from Lemma \ref{transg} that $\ind(V \ast_p \Z/p)=3$. We leave it to the reader to check that $\ind(V)$ is also 3.
\end{proof}

\begin{remark}
The complexes $\tilde V$ were found in the following way: they are the simplest examples of a space of index $k$
 such that the generator of $H_{k+1} B\Z/p$ does not transgress in the Serre spectral sequence of
 $\tilde V \to V \to B\Z/p$. This can be shown using the cellular model for the Borel construction described in \cite{AP}.
\end{remark}

\end{document}